\documentclass[article,11pt,reqno]{amsart}
\usepackage[english]{babel}
\usepackage[utf8]{inputenc}
\usepackage{thm-restate}
\usepackage{amsmath, amssymb, amsthm}
\usepackage[table]{xcolor}
\usepackage{tikz-cd}
\usepackage{graphics}
\usepackage[]{todonotes}
\usepackage{physics}
\usepackage{setspace}
\usepackage[margin=1.25in]{geometry}
\usepackage[hidelinks]{hyperref}
\usepackage{wrapfig}
\usepackage{caption}
\usepackage{subcaption}
\usepackage{floatrow}
\usepackage{xpatch}

\usepackage{graphicx}
\usepackage{mathrsfs}
\usepackage{verbatim}
\usepackage{adjustbox}
\newcommand{\nocontentsline}[3]{}
\newcommand{\tocless}[2]{\bgroup\let\addcontentsline=\nocontentsline#1{#2}\egroup}

\DeclareMathOperator*{\UEG}{UEG}
\DeclareMathOperator*{\UOG}{UOG}

\DeclareMathOperator{\sd}{sd}
\DeclareMathOperator{\hp}{hp}

\newcommand{\Z}{\mathbb{Z}}

\newcommand{\cc}{\leftrightarrow}

\newcommand{\T}{\mathbb T}

\usepackage{tikz}
\usepackage{pgfplots}
\usepackage{tikz-3dplot} 
\usetikzlibrary{positioning}
\usetikzlibrary{arrows}

\newcommand{\n}[1]{ \left \vert \left \vert #1 \right \vert  \right \vert}

\newtheorem{theorem}{Theorem}[section]

\newtheorem{lemma}[theorem]{Lemma}

\newtheorem{corollary}[theorem]{Corollary}

\newtheorem{question}[theorem]{Question}
\newtheorem{example}[theorem]{Example}
\newtheorem{proposition}[theorem]{Proposition}

\usepackage{cleveref}
\usetikzlibrary{positioning,calc}

\makeatletter
\newcommand{\changeoperator}[1]{%
  \csletcs{#1@saved}{#1@}%
  \csdef{#1@}{\changed@operator{#1}}%
}
\newcommand{\changed@operator}[1]{%
  \mathop{%
    \mathchoice{\textstyle\csuse{#1@saved}}
               {\csuse{#1@saved}}
               {\csuse{#1@saved}}
               {\csuse{#1@saved}}%
  }%
}
\makeatother

\changeoperator{prod}
\pgfplotsset{compat=1.17}

\title{Half-plane non-coexistence without FKG}

\author{Frederik Ravn Klausen}
\address{Frederik Ravn Klausen, University of Cambridge, DPMMS, Cambridge, United Kingdom}
\email{frk23@cam.ac.uk}

\author{Noah Kravitz}
\address{Noah Kravitz, St John's College, Oxford and Mathematical Institute, University of Oxford; St Giles', Oxford OX1 3JP, UK}
\email{noah.kravitz@maths.ox.ac.uk}

\begin{document}

\begin{abstract}
For $\mu$ an edge percolation measure on the infinite square lattice, let $\mu_{\texorpdfstring{\hp}{hp}}$ (respectively, $\mu^*_{\texorpdfstring{\hp}{hp}}$) denote its marginal (respectively, the marginal of its planar dual process) on the upper half-plane.  We show that if $\mu$ is translation-invariant and ergodic and almost surely has only finitely many infinite clusters, then either almost surely $\mu_{\texorpdfstring{\hp}{hp}}$ has no infinite cluster, or almost surely $\mu^*_{\texorpdfstring{\hp}{hp}}$ has no infinite cluster. 
By the classical Burton--Keane argument, these hypotheses are satisfied if $\mu$ is translation-invariant and ergodic and has finite-energy. 
In contrast to previous ``non-coexistence'' theorems, our result does not impose a positive-correlation (FKG) hypothesis on $\mu$.  
Our arguments also apply to the random-cluster model (including the regime $q<1$, which lacks FKG), the uniform spanning tree, and the uniform odd subgraph.
\end{abstract}

\maketitle
\section{Introduction}
\subsection{Background}
This note concerns non-coexistence phenomena for edge (bond) percolation on  the square lattice $\Z^2$ and the upper half-plane $\Z \times \Z_{\geq 0}$. On the full plane $\Z^2$, a celebrated argument of Zhang shows that (almost surely) infinite primal and dual clusters cannot coexist in any translation-invariant, finite-energy, and positively-associated (FKG) edge percolation measure.  See \cite[page 289]{grimmett2012percolation} for an exposition and \cite{bollobas2008percolation, glazman2025planar}, \cite[Theorem 9.3.1]{sheffield2005random} for generalizations.
Here, as in many arguments in percolation theory, the positive-association hypothesis is essential, and without this condition there can be pathological examples (as in, e.g., \cite{haggstrom2009some}).  
It is nonetheless desirable to see what can be said in the absence of positive-association, since many planar percolation models of interest, such as the loop $O(n)$ model, the random-cluster model with $q<1$, the arboreal gas, and other constrained percolation models, do not satisfy the FKG inequality.

Translation-invariance alone is certainly not sufficient to guarantee non-coexistence. For example, the dual of the uniform spanning tree is again a uniform spanning tree, so for this measure both the primal and the dual have infinite clusters \cite{PemantleUST}.  These infinite clusters are ``non-robustly connected'' in the sense that each vertex has a unique path to infinity.  A more troubling example is Häggström and Mester's construction \cite{haggstrom2009some} of a translation-invariant, finite-energy site percolation measure on $\Z^2$ with simultaneous infinite clusters of open and closed vertices which are robust under iid thinning.

In the present paper we show that the situation is better in the half-plane than in the full plane: Here, translation-invariance, finite-energy, and absence of infinitely many infinite clusters together suffice for non-coexistence.

\subsection{Main results}

We work with the square lattice $\Z^2$, namely, the Cayley graph of $\mathbb{Z}^2$ with respect to the generators $(1,0),(0,1)$.
Subgraphs of $\Z^2$ correspond to tuples $\omega \in  \{0,1\}^{E(\Z^2)}$,  where a coordinate $1$ represents an ``open'' edge that is present and a coordinate $0$ represents a ``closed'' edge that is absent. By an \emph{edge percolation measure} we mean a probability measure on $\{0,1\}^{E(\Z^2)}$ (with the product topology).

We need a few more notions before we can state our main results.  The natural translation action of $\mathbb{Z}^2$ on the square lattice extends to an action of $\mathbb{Z}^2$ on $\{0,1\}^{E(\Z^2)}$; we say that a percolation measure is \emph{translation-invariant} if it is invariant under this action.  Say that an edge percolation measure $\mu$ has \emph{finite-energy} if there is some $\delta>0$ such that for every edge $e \in E(\Z^2)$, the conditional probability $$\mu(\omega_e=1|\omega_{E(\Z^2) \setminus\{e\}})$$  has essential infimum at least $\delta$ and essential supremum at most $1-\delta$; this condition, which is sometimes called uniform finite-energy in the literature, is stronger than other common notions of finite-energy which have less uniformity (in either the edge $e$ or the bounds on the conditional probability).

Every edge percolation measure $\mu$ on $\Z^2$ induces an edge percolation measure $\mu^*$ on the planar dual of $\Z^2$, which is itself isomorphic to $\Z^2$; see \Cref{sec:duality} below for further discussion of duality.  It is easy to see that $\mu$ is ergodic (respectively, has finite-energy) if and only if $\mu^*$ is ergodic (respectively, has finite-energy).  Finally, we write $\mu_{\hp}$ for the marginal of $\mu$ on the upper half-plane $\Z \times \Z_{ \geq 0}$, and we define $\mu^*_{\hp}$ analogously.  Our first non-coexistence result goes as follows.

\begin{theorem}\label{thm:main}
Let $\mu$ be a translation-invariant, ergodic, finite-energy edge percolation measure on $\Z^2$.  Then one of the following holds:
\begin{itemize}
    \item Almost surely $\mu_{\hp}$ has a unique infinite cluster and $\mu^*_{\hp}$ has no infinite cluster.
    \item Almost surely $\mu^*_{\hp}$ has a unique infinite cluster and $\mu_{\hp}$ has no infinite cluster.
    \item Almost surely neither $\mu_{\hp}$ nor $\mu^*_{\hp}$ has an infinite cluster.
\end{itemize}
\end{theorem}

If we drop the ergodicity hypothesis from this theorem, then we still obtain the conclusion that for $H$ sampled from $\mu_{\hp}$, almost surely the number of infinite clusters in $H$ plus the number of infinite clusters in $H^*$ is at most $1$; in particular, almost surely $H$ and $H^*$ do not simultaneously have infinite clusters.  This statement follows from applications of the theorem to the individual components in the ergodic decomposition of $\mu$, which almost surely inherit the finite-energy property (see, e.g., \cite[page 307]{burton1991topological}).

Classical work of Burton and Keane~\cite{burton1989density} shows that under the hypotheses of \Cref{thm:main}, the number of infinite clusters in $\mu$ is almost surely at most $1$.\footnote{Burton and Keane's uniqueness argument has also been adapted to several settings that fall slightly short of having  finite-energy; examples include the loop $O(n)$ model \cite[Proposition 4.11]{crawford2020macroscopic},  the random current representation of the Ising model \cite{aizenman2015random}, various constrained percolation models \cite{lima2020constrained,holroyd2021constrained}, and the 1-2 model \cite{li2014uniqueness}.}  Thus \Cref{thm:main} is an immediate consequence of the following theorem, our main technical result.

\begin{theorem}\label{thm:main2}
Let $\mu$ be a translation-invariant, ergodic edge percolation measure on $\Z^2$ which has finitely many infinite clusters almost surely.  Then the trichotomy of \Cref{thm:main} holds.
\end{theorem}

\Cref{thm:main2} applies to some models, such as the uniform spanning tree, which \Cref{thm:main} does not cover.  We remark that there do exist translation-invariant edge percolation measures on $\mathbb{Z}^2$ that almost surely have multiple, but only finitely many, infinite clusters.\footnote{For an example that almost surely has $2$ infinite clusters, one can modify the uniform spanning tree as follows.  Sample a uniform spanning tree $H$ on the dilated infinite square grid $(2\mathbb{Z})^2$.  Let $H^*$ be its planar dual (also a spanning tree), interpreted as a subgraph of the square grid $(1+2\mathbb{Z})^2$.  By subdividing each edge in half, we can interpret $H \cup H^*$ as a subgraph of $\mathbb{Z}^2$; let $\mu'$ denote the resulting edge percolation measure on $\Z^2$.  Notice that $\mu'$ is invariant under translations by $(2\Z)^2$ and always has exactly $2$ infinite clusters.  Finally, let $\mu$ be the average of $\mu'$ and its translates by $(1,0),(0,1),(1,1)$; this guarantees that $\mu$ is $\mathbb{Z}^2$-translation-invariant, and of course it still always has $2$ infinite clusters.  One can modify this construction to obtain a larger finite number of infinite clusters. See also the similar construction in \cite{garcia2026connected}, where the possible number of ends is also discussed. }  After the present paper appeared in preprint, Tim\'ar~\cite{timar2026parabolic} generalized \Cref{thm:main2} and showed that half-plane coexistence is a criterion that distinguishes parabolicity from hyperbolicity.

We record the following corollary of \Cref{thm:main2} for self-dual measures.

\begin{corollary}\label{cor:self-dual}
Let $\mu$ be a self-dual, translation-invariant, ergodic edge percolation measure on $\Z^2$ that almost surely has only finitely many infinite clusters.  Then $\mu_{\hp}$ almost surely has no infinite cluster.

In particular, if $\mu$ is a self-dual, translation-invariant, ergodic, finite-energy edge percolation measure on $\Z^2$, then $\mu_{\hp}$ almost surely has no infinite cluster.
\end{corollary}

The proof strategy for Theorem~\ref{thm:main2} has two steps.  The first step is showing that the number of infinite clusters in $\mu_{\hp}$ is either almost surely $0$ or almost surely $1$.  The main challenge is eliminating the possibility that an infinite cluster in the full plane ``splits'' into several infinite clusters in the half-plane.  The second step is showing that if $\mu_{\hp}$ almost surely has a unique infinite cluster, then almost surely $\mu^*_{\hp}$ has no infinite cluster; it suffices to show that almost surely the plaquette centered at $(1/2,1/2)$ does not connect to infinity in the dual.  Here, almost surely both the positive $x$-axis and the negative $x$-axis connect to the unique infinite primal cluster, and these paths to infinity together ``block'' the origin plaquette from connecting to infinity in the dual.  A recurring technical point is ruling out the existence of what we term ``tenuous'' infinite clusters.

Our methods are robust enough to establish some slight extensions of these theorems.  Instead of taking the upper half-plane, one could take the half-plane lying on one side of any given line with rational slope.  Also, one could replace the square lattice with a different planar graph admitting a $\mathbb{Z}^2$-translation action with only finitely many orbits.

\subsection{Applications}

Let us highlight four special cases of our results.  First, when $\mu$ is Bernoulli percolation with density $1/2$, Corollary~\ref{cor:self-dual} (via finite-energy) provides a new proof of the fact that almost surely there is no infinite cluster in the half-plane.  Of course it is a classical result of Harris~\cite{harris1960lower} that there is no infinite cluster even in the full plane; the novelty of our argument is that it avoids using FKG (or the Harris inequality).  See also~\cite[Theorem 5]{BS} for another Bernoulli-$1/2$ half-plane non-percolation result of a similar flavor.

Second, Theorem~\ref{thm:main} demonstrates non-coexistence in the half-plane for (a suitable limit of) the random-cluster model with all values of the parameter $q$; see \Cref{sec:rc} below for precise definitions.  The random-cluster model has FKG only when $q \geq 1$, and for $q<1$ little is known.

Third, the uniform spanning tree on $\Z^2$ is self-dual and ergodic (due to tail-triviality; see \cite[Theorem 10.18]{lyons2017probability}), and it has a unique infinite cluster. 
Thus, Corollary~\ref{cor:self-dual} shows that the half-plane marginal of the uniform spanning tree almost surely has only finite clusters.  (Note that this model lacks finite-energy.).  The same holds for the Peano uniform spanning tree (UST) curve \cite{lyons2017probability} and for constructions of Häggström--Mester type \cite{haggstrom2009some} (if they can be adapted to edge percolation).

Fourth, we show that the uniform odd subgraph almost surely has no infinite cluster in the half-plane.  This does not follow from the statement of Corollary~\ref{cor:self-dual}, but our arguments from Theorem~\ref{thm:main2} can be adapted.  The uniform odd subgraph measure is self-complementary, and here the ``non-crossing'' property of the primal and the complement plays the role that planar duality plays in Theorem~\ref{thm:main2}.

\subsection{Organization of the paper}
We start by proving \Cref{thm:main2} in Section~\ref{sec:thm1.1}.  We then turn to applications: Section~\ref{sec:rc} contains our results on the $q<1$ random-cluster model, and Section~\ref{sec:odd} treats the uniform odd subgraph.  In Section~\ref{sec:conclusion} we describe several open problems and pathological examples in the full plane.  \Cref{sec:appendix} contains an alternative proof of a variant of one of our technical lemmas (the absence of tenuous infinite clusters).

\section{Proof of the main theorem}\label{sec:thm1.1}
In this section we prove \Cref{thm:main2} in two steps as outlined in the introduction.  Before describing the second step we briefly review planar duality.  We record here a (standard) lemma that will be of use several times.

\begin{lemma}\label{lem:translation-invariant-events}
Let $\mu$ be an edge percolation measure on $\Z^2$ that is invariant and ergodic under horizontal translations.  Let $P$ be a measurable event with $\mu(P)>0$, and let $P(i)$ denote the event that $P$ holds after translating horizontally by $i$.  Then almost surely there are both arbitrarily large (positive) and small (negative) $i$'s for which $P(i)$ occurs.
\end{lemma}

\begin{proof}
Consider the set of indices $$S:=\{i \in \Z: \text{$P(i)$ occurs}\}.$$
We show that almost surely $\sup S=\infty$; the argument for $\inf S=-\infty$ is identical.  The event $S\neq \emptyset$ is translation-invariant, and since $S$ contains $0$ with positive probability, ergodicity guarantees that $S$ is almost surely nonempty.  Horizontal translation-invariance implies that $\sup S$ assumes all integer values with equal probability, so this probability must be $0$.  Thus almost surely $\sup S =\infty$.
\end{proof}

Of course, one could use an ergodic theorem (von Neumann's mean ergodic theorem or Birkhoff's pointwise ergodic theorem, for instance) to obtain quantitative versions of Lemma~\ref{lem:translation-invariant-events}, but the above soft version suffices for us.

\subsection{Uniqueness of the half-plane infinite cluster}\label{sec:uniqueness}
We establish the uniqueness of the infinite cluster in $\mu_{\hp}$ (if any exists)
under the assumptions of \Cref{thm:main2}.  The main part of the argument relies on planar topology and bears some resemblance to the approach of Zhang in his non-coexistence proof.

\begin{lemma}\label{lemma:every_cluster_is_everywhere}
Let $\mu$ be a translation-invariant edge percolation measure on $\Z^2$ that almost surely has only finitely many infinite clusters. Then almost surely each infinite cluster contains vertices with arbitrarily large (positive) and arbitrarily small (negative) $y$-coordinates. 
\end{lemma}

\begin{proof}
By performing an ergodic decomposition, we may assume that the number of infinite clusters in $\mu$ is almost surely equal to some $N \in \mathbb{Z}_{\geq 0}$. 
There is nothing to show if $N=0$, so assume that $N \geq 1$.  Consider the supremal $y$-coordinate that each cluster meets, and let $X \subseteq \Z \cup \{\infty\}$ denote the multiset consisting of these $y$-coordinates. Almost surely $X$ is an $N$-element multiset.  Translation-invariance implies that $\min X$ assumes all integer values with equal probability, so this probability must be $0$.  Thus almost surely $\min X=\infty$, i.e., each cluster contains vertices with arbitrarily large $y$-coordinates.  Likewise, almost surely each cluster contains vertices with arbitrarily small $y$-coordinates.
\end{proof}

Say that an infinite graph has at least $k$ \emph{ends} if it has a finite vertex subset whose deletion disconnects the graph into at least $k$ infinite connected components.  A generalization of the Burton--Keane argument shows (see, e.g., \cite[Exercise 7.24]{lyons2017probability}) that for any translation-invariant edge percolation measure on $\Z^2$ (indeed, on any transitive amenable graph), almost surely all infinite clusters have at most $2$ ends.

\begin{lemma}\label{lemma:one_end_to_no_clust}
Let $\mu$ be a translation-invariant edge percolation measure on $\Z^2$ that almost surely has only finitely many infinite clusters.  Then $\mu_{\hp}$ almost surely has at most $1$ infinite cluster and has no infinite cluster that avoids touching the $x$-axis.
\end{lemma}
\Cref{ex:stacked_random_walks} below shows that the converse is not true: It is possible for $\mu$ to have infinitely many infinite clusters in the full plane but no infinite clusters in the half-plane.

\begin{proof}
Let us take an ergodic decomposition  $\mu = \int \nu \,d \lambda(\nu)$ with respect to horizontal translations, so that each ergodic component $\nu$ is invariant under horizontal translations.  For $\lambda$-almost every $\nu$, there are $N=N(\nu) \in \Z_{\geq 0}$ and $k=k(\nu) \in \Z_{\geq 0} \cup \{\infty\}$ such that almost surely $\nu$ has $N$ infinite clusters and $\nu_{\hp}$ has $k$ infinite clusters.  Fix some such $\nu$.  We first show that $k \leq 2N+1$ and then improve this bound to $k \leq 1$.

Assume for the sake of contradiction that $k \geq 2N+2$.  Continuity of measure provides some $n \in \Z_{\geq 1}$ such that with positive probability, the marginal of $\mu$ on $$(\Z \times \Z_{\geq 0}) \setminus ([-n,n] \times [0,n])$$ has infinite (half-plane) clusters $C_1,C_2, \ldots, C_{2N+2}$ that touch the boundary of the box $[-n,n] \times [0,n]$ (there may be additional clusters but they will not play a role). For ease of notation, label these clusters from left to right according to the order in which they touch the boundary of $[-n,n] \times [0,n]$.  Now consider the picture in the full plane.  The clusters $C_1$ and $C_{2N+2}$ may glue together via a path that passes into the lower half-plane, but the other clusters $C_2, \dots, C_{2N+1}$ must all remain distinct, since $C_1$, $C_{2N}$, and the box together ``block'' these clusters from reaching the $x$-axis.  See \Cref{fig:oneend3}.  Thus, with positive probability $\mu$ has $N$ clusters with a total of at least $2N+1$ ends.  In this case the pigeonhole principle shows that some cluster has at least $3$ ends, in contradiction with the fact stated before the lemma.

So far, we know that $k(\nu) \leq 2N(\nu)+1<\infty$ for $\lambda$-almost every $\nu$.  \Cref{lemma:every_cluster_is_everywhere} gives
$$\mu[\text{all infinite clusters touch the $x$-axis}] =1.$$ 
It follows that for $\lambda$-almost every $\nu$, we also have
$$\nu[\text{all infinite clusters touch the $x$-axis}] =1,$$
and in particular almost surely all infinite clusters in $\nu_{\hp}$ touch the $x$-axis.  For such a $\nu$, horizontal translation-invariance guarantees that each infinite cluster contains the origin with positive probability.  Then \Cref{lem:translation-invariant-events} ensures that almost surely each of the infinite clusters in fact touches the $x$-axis at both arbitrarily large and arbitrarily small $x$-coordinates.  But this can be the case for at most a single cluster by planarity (distinct clusters cannot cross), so $k(\nu) \leq 1$.  Integrating over $\nu$ yields the desired conclusion for $\mu$.
\end{proof}

\begin{figure}
    \centering
    \includegraphics[width=0.5\linewidth]{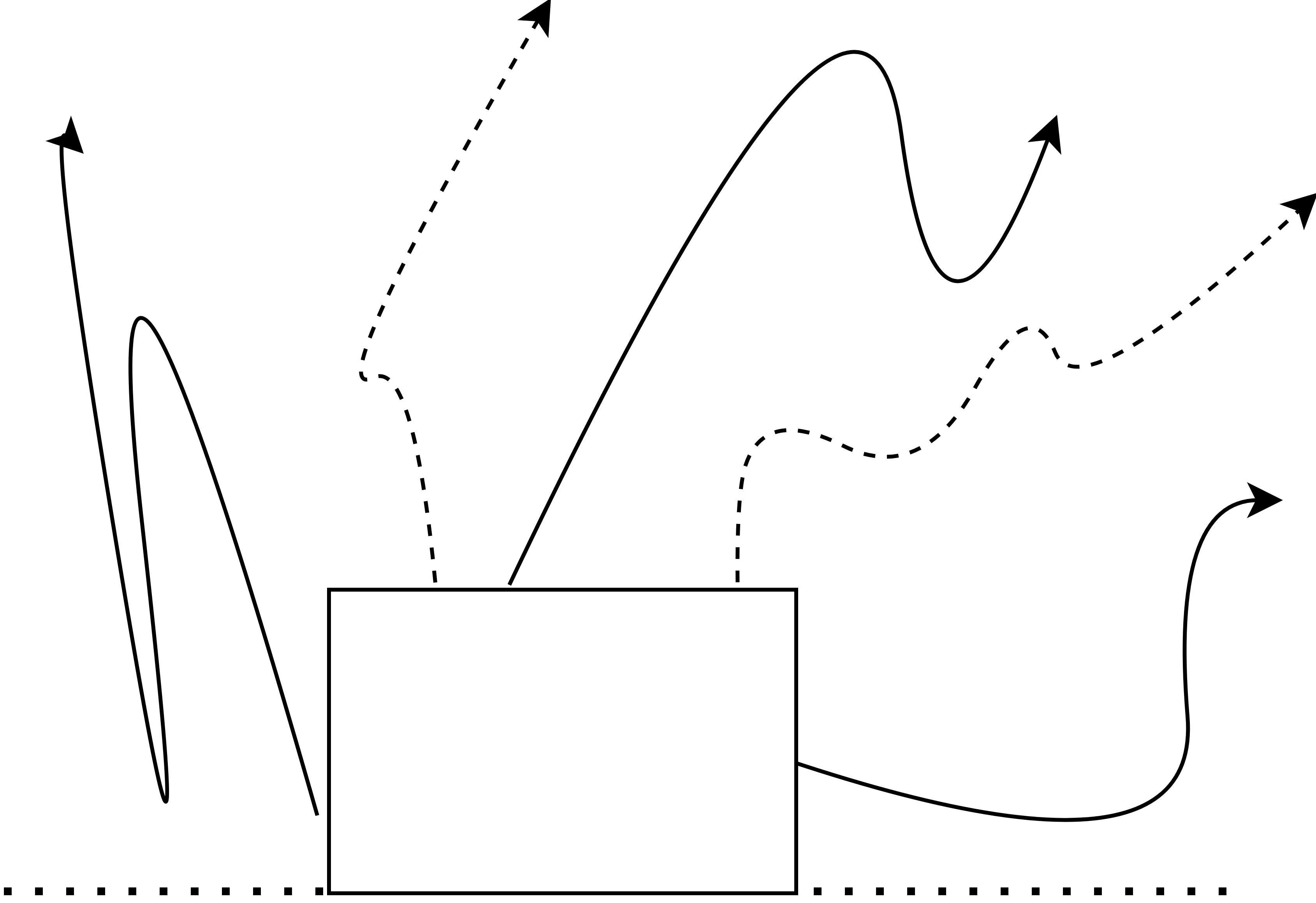}
    \caption{If three infinite clusters (solid lines) in the half-plane touch the box, then they are separated by two dual clusters (dotted lines).  The middle cluster cannot glue to any other cluster when we reveal the lower-half-plane.}
    \label{fig:oneend3}
\end{figure}

We would like to know slightly more, namely, that the number of infinite clusters in $\mu_{\hp}$ is either almost surely $0$ or almost surely $1$.  In the full plane, ergodicity immediately implies that the number of infinite clusters in $\mu$ is almost surely equal to some constant in $\mathbb{Z}_{\geq 0} \cup \{\infty\}$.  In the half-plane, however, obtaining the analogous statement is trickier because vertical translations do not necessarily preserve the number of infinite clusters.

To show that the number of infinite clusters in $\mu_{\hp}$ is almost surely constant, we need to rule out the degenerate case of an infinite cluster whose ``infinite-ness'' depends wholly on a neighborhood of the $x$-axis.   To this end, say that an infinite cluster $C$ in $\Z \times \Z_{\geq 0}$ is \emph{$n$-tenuous} if the intersection of $C$ with the half-plane $\Z \times \Z_{\geq n}$ has no infinite connected components; say that $C$ is \emph{tenuous} if it is $n$-tenuous for some $n \in \Z_{\geq 1}$.

\begin{proposition}\label{lemma:uniqueness_implies_non_tenuousness}
    Let $\mu$ be a translation-invariant edge percolation measure on $\Z^2$ such that $\mu_{\hp}$ almost surely has at most $1$ infinite cluster.  Then $\mu_{\hp}$ almost surely has no tenuous infinite cluster. 
\end{proposition}

\begin{proof}
Fix $n \in \mathbb{Z}_{\geq 1}$.  Let $I_0$ be the event that (a subgraph sampled according to) $\mu_{\hp}$ has an infinite cluster, and let $I_n$ be the event that the marginal of $\mu_{\hp}$ on $\mathbb{Z} \times \mathbb{Z}_{\geq n}$ has an infinite cluster.  Note that $I_n \subseteq I_0$.  Vertical translation-invariance gives $\mu_{\hp}[I_0]=\mu_{\hp}[I_n]$ and hence $\mu_{\hp}[I_0 \setminus I_n]=0$.  But $I_0 \setminus I_n$ is precisely the event that $\mu_{\hp}$ has an $n$-tenuous infinite cluster.  A union bound over $n$ establishes the proposition.
\end{proof}

In \Cref{sec:appendix} we provide a ``combinatorial'' proof of the similar statement that if $\mu$ is a finite-energy edge percolation measure on $\Z^2$, then almost surely $\mu_{\hp}$ has no tenuous infinite cluster.  This alternative proof is not logically necessary for the following arguments, but we think that it may contain ideas of independent interest.

From Proposition~\ref{lemma:uniqueness_implies_non_tenuousness} we quickly deduce that the number of infinite clusters in $\mu_{\hp}$ is almost surely constant, and thus either almost surely $0$ or almost surely $1$.

\begin{lemma}\label{lem:a.s.constant}
Let $\mu$ be a translation-invariant, ergodic edge percolation measure on $\Z^2$ that almost surely has only finitely many infinite clusters.  Then the number of infinite clusters in $\mu_{\hp}$ is either almost surely $0$ or almost surely $1$.
\end{lemma}

\begin{proof}
\Cref{lemma:one_end_to_no_clust} tells us that $\mu_{\hp}$ almost surely has at most $1$ infinite cluster.  To complete the proof it suffices to show that the number of infinite clusters is almost surely constant. 

\Cref{lemma:uniqueness_implies_non_tenuousness} tells us that almost surely $\mu_{\hp}$ has no tenuous infinite cluster.  The number of non-tenuous infinite clusters in $\mu_{\hp}$ is a measurable function valued in $\Z_{\geq 0} \cup \{\infty\}$.   It is invariant under horizontal translations and (due to non-tenuousness!) non-decreasing under downward vertical translations (see \Cref{fig:tenuous}), so by the ergodicity of $\mu$ it is almost surely constant.  Indeed, it suffices to check that each super-level set has measure either $0$ or $1$.  For each $n \in \Z_{\geq 0}$, let $A_n$ be the event that there are at least $n$ non-tenuous infinite clusters in the upper half-plane.  We know that horizontal translations preserve $A_n$.  Downward vertical translations map $A_n$ into itself and in particular preserve $A_n$ up to sets of measure $0$; thus all vertical translations preserve $A_n$ up to sets of measure $0$.  Let $A^*_n$ be the intersection of all of the translations of $A_n$; this is clearly an invariant set, and $\mu(A^*_n)=\mu(A_n)$ since each of the countably many translates of $A_n$ differs from $A_n$ in a set of measure $0$.  Ergodicity implies that $\mu(A^*_n) \in \{0,1\}$, so the same holds for $\mu(A_n)$. 
\end{proof}

\begin{figure}
\begin{tikzpicture}[scale=1.0, line width=1pt]

\def\strip{0.9} 

\begin{scope}[xshift=0cm]
  \fill[gray!20] (-0.5,0) rectangle (2.2,\strip);
  \draw (-0.5,0) -- (2.2,0);
  \draw[very thick] (-0.5,\strip) -- (2.2,\strip);
  \draw plot[smooth] coordinates {(-0.5,0.3) (0.3,0.2) (0.6,0.6) (1,2.5)};
  \node at (1.0,-0.8) {same};
\end{scope}

\begin{scope}[xshift=4cm]
  \fill[gray!20] (-0.5,0) rectangle (2.2,\strip);
  \draw (-0.5,0) -- (2.2,0);
  \draw[very thick] (-0.5,\strip) -- (2.2,\strip);
  \draw plot[smooth] coordinates {(0,2.5) (0.8,0.4) (1.6,2.5)};
  \node at (1.0,-0.8) {increase};
\end{scope}

\begin{scope}[xshift=8cm]
  \fill[gray!20] (-0.5,0) rectangle (2.2,\strip);
  \draw (-0.5,0) -- (2.2,0);
  \draw[very thick] (-0.5,\strip) -- (2.2,\strip);
  \draw plot[smooth] coordinates {(-0.5,0.6) (-0.2,0.6) (0.2,1.2) (0.6,0.6) (1.0,1.5) (1.4,0.5) (1.8,1.3) (2.2,0.65)};
  \node at (1.0,-0.8) {decrease};
\end{scope}

\end{tikzpicture}

\caption{Removing a finite-height vertical strip can change the number of infinite clusters.  The first two examples show non-tenuous infinite clusters, and the third shows a tenuous infinite cluster.  \label{fig:tenuous}}
\end{figure}
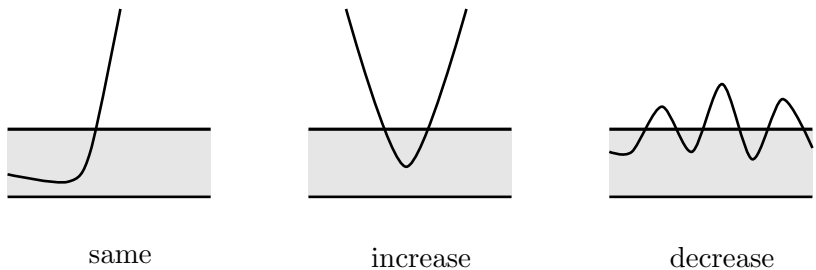

\subsection{Planar duality}\label{sec:duality}

Let $G$ be a planar graph with a fixed planar embedding.  The \emph{planar dual} of $G$, denoted $G^*$, is the graph whose vertex set is the set of plaquettes (faces) of $G$, where each edge of $G$ corresponds to an edge in $G^*$ connecting the two plaquettes touching that edge.  Notice that $E(G),E(G^*)$ are naturally identified.  The graph $\Z^2$ is \emph{self-dual} in the sense that its dual is isomorphic to $\Z^2$ (identify each vertex $(x,y)$ of $\Z^2$ with the plaquette centered at $(x+1/2,y+1/2)$).

The \emph{complement} of a subgraph of $G$ is the subgraph obtained by exchanging open and closed edges, i.e., the complement of $(\omega_e)_{e \in E(G)}$ is $(1-\omega_e)_{e \in E(G)}$. 
Say that the \emph{dual} of a subgraph $H$ of $G$ is the subgraph $H^*$ of $G^*$ with edge set corresponding to the edge set of the complement of $H$.  Notice that the edges of $H$ and $H^*$ do not cross.  Thus, for $G=\Z^2$, a cluster of $H$ is finite if and only if it is enclosed by a loop of $H^*$, and vice versa.

The dual operation also lets us assign a dual measure $\mu^*$ on $G^*$ to any edge percolation measure $\mu$ on $G$.  It is easy to show that $\mu$ is translation-invariant (respectively, has finite-energy) if and only if $\mu^*$ is translation-invariant (respectively, has finite-energy).  In the case $G=\Z^2$, we write $\mu^*_{\hp}$ as shorthand for $(\mu^*)_{\hp}$.

We will need 
the following simple fact, for which we could not find a reference.
\begin{lemma}\label{lemma:iff_infinite}
    Let $\mu$ be a translation-invariant edge percolation measure on $\Z^2$. Then $\mu$ has infinitely many infinite clusters with positive probability if and only if $\mu^*$ has infinitely many infinite clusters with positive probability.
\end{lemma}
\begin{proof}
It suffices to prove the forward implication.  Assume that $\mu$ has infinitely many infinite clusters with probability $\gamma>0$.  Let $N \in \Z_{\geq 1}$.  Continuity of measure provides some large box $B=B(N)$ that with probability at least $\gamma/2$ intersects at least $2N$ infinite clusters of $\mu$.  With the same probability it is moreover the case that all infinite clusters of $\mu^*$ have at most $2$ ends (see the remark before the statement of \Cref{lemma:one_end_to_no_clust}).
    
In $\Z^2 \setminus B$, the boundary of $B$ alternately touches infinite primal and dual ``arms'' (clusters with paths to infinity).  Thus, whenever there are at least $2N$ primal arms, there are also at least $2N$ dual arms.  In particular, the infinite dual clusters in $\Z^2$ have at least $2N$ ends in total, so there are at least $N$ infinite dual clusters.  Since this occurs with probability at least $\gamma/2$ for each $N$, we conclude that $\mu^*$ has infinitely many infinite clusters with probability at least $\gamma/2$.
\end{proof}

\subsection{Non-coexistence}
We showed in \Cref{sec:uniqueness} that under the hypotheses of \Cref{thm:main2}, the number of infinite clusters in $\mu_{\hp}$ is either almost surely $0$ or almost surely $1$, and in the latter case the infinite cluster is almost surely not tenuous; by \Cref{lemma:iff_infinite} the same also holds for $\mu^*_{\hp}$.  This gives four possibilities, and \Cref{thm:main2} amounts to the assertion that $\mu_{\hp}$ and $\mu^*_{\hp}$ do not both almost surely have unique infinite clusters.  We achieve this in the following proposition.

\begin{proposition}\label{proposion31}
    Let $\mu$ be a translation-invariant edge percolation measure on $\Z^2$ which almost surely has only finitely many infinite clusters, and assume that $\mu_{\hp}$ almost surely has a unique infinite cluster.
    Then almost surely $\mu^*_{\hp}$ has no infinite cluster. 
\end{proposition}

\begin{proof}
\Cref{lemma:one_end_to_no_clust} tells us that the (unique) infinite cluster of $\mu_{\hp}$ almost surely touches the $x$-axis. Horizontal translation-invariance then implies that the origin is contained in an infinite cluster of $\mu_{\hp}$ with positive probability. 
The same argument (combined with \Cref{lemma:iff_infinite}) shows that if $\mu^*_{\hp}$ contains an infinite cluster with positive probability, then with positive probability the dual lattice point $(1/2,1/2)$ is contained in an infinite cluster.  Thus, to prove the proposition it suffices to show that $\mu^*_{\hp}[(1/2,1/2) \cc \infty] = 0$.

Perform an ergodic decomposition $\mu = \int \nu \,d \lambda(\nu)$ with respect to horizontal translations.  Then $\lambda$-almost surely, $\nu_{\hp}$ has a unique infinite cluster and this cluster touches the $x$-axis. Horizontal translation invariance implies that the origin is contained in an infinite cluster of $\nu_{\hp}$ with positive probability, and Lemma~\ref{lem:translation-invariant-events} tells us that almost surely there are some $x_1\leq 0$ and $x_2 \geq 1$ such that each $(x_i,0)$ is contained in the (unique) infinite cluster.   In this case there must be a path $P$ connecting the vertices $(x_1,0)$ and $(x_2,0)$.  The path $P$ and the $x$-axis together form a loop\footnote{There is almost surely a path $P$ that passes above $(1/2,1/2)$. Indeed, if not, $\mu_{\hp}$ would have a $1$-tenuous infinite cluster, which is ruled out by \Cref{lemma:uniqueness_implies_non_tenuousness}.} that prevents $(1/2,1/2)$ from being contained in an infinite cluster in $\nu^*_{\hp}$; see Figure~\ref{fig:blocking1}. Thus $\nu^*_{\hp}[(1/2,1/2) \cc \infty]=0$.  Integrating over $\nu$ gives $\mu^*_{\hp}[(1/2,1/2) \cc \infty] = 0$, as desired. 

\begin{figure}[ht]
    \includegraphics[scale = 0.05]{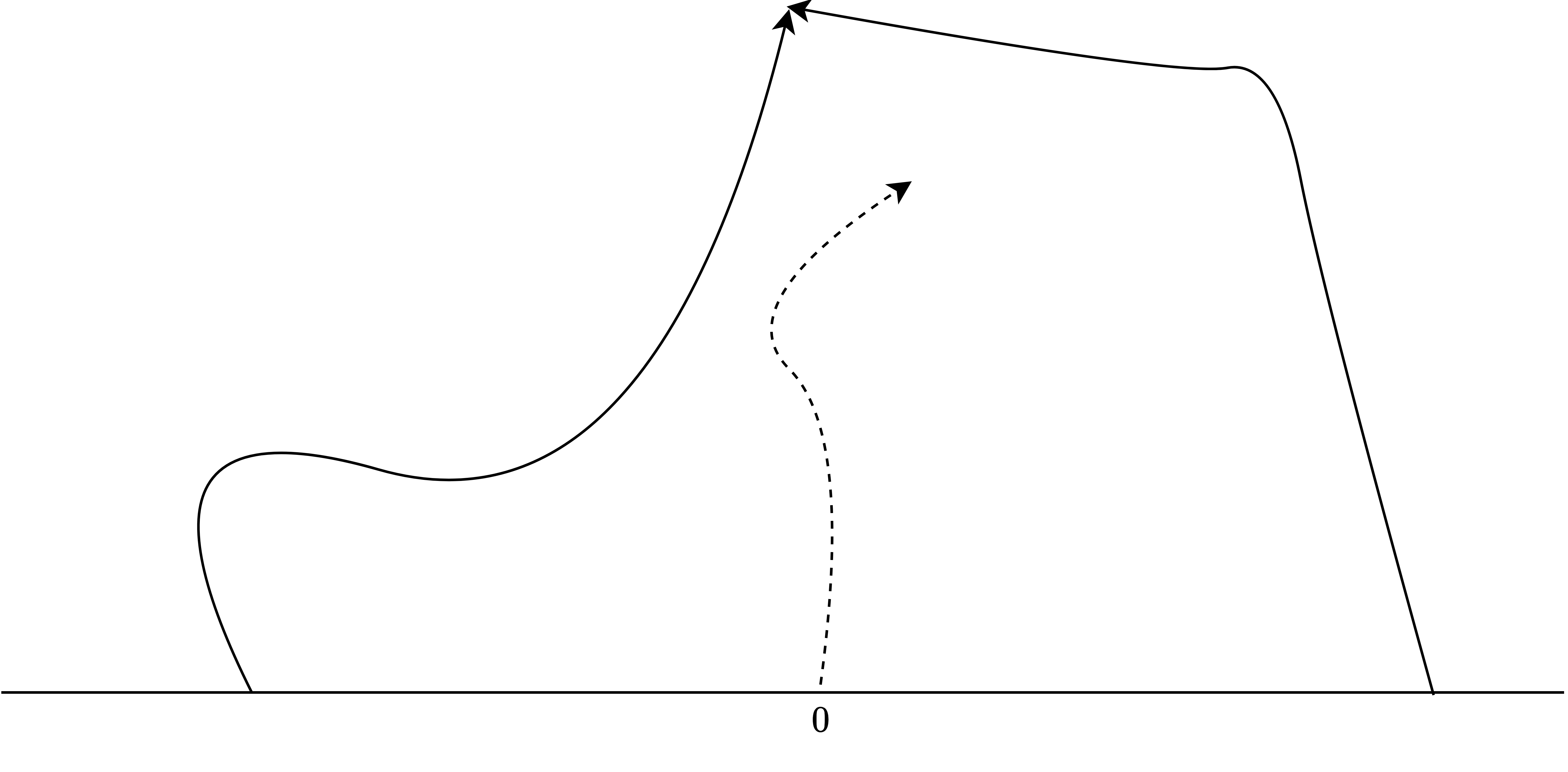}
    \caption{ \label{fig:blocking1} The solid lines connect the positive and negative $x$-axes to the unique infinite primal cluster.  This creates a finite region which traps the dual cluster containing the origin plaquette.}
\end{figure}
\end{proof}

\begin{proof}[Proof of \Cref{thm:main2}]
\Cref{lem:a.s.constant} tells us that the number of infinite clusters in $\mu_{\hp}$ is either almost surely $0$ or almost surely $1$, and by \Cref{lemma:iff_infinite} the same holds for $\mu^*_{\hp}$.  If $\mu_{\hp}$ almost surely has no infinite cluster, then either the second or third outcome of the trichotomy occurs.  If instead $\mu_{\hp}$ almost surely has a unique infinite cluster, then \Cref{proposion31} implies that almost surely $\mu^*_{\hp}$ has no infinite cluster, so the first outcome of the trichotomy occurs.
\end{proof}

We remark that if one assumes ergodicity under horizontal translations (rather than merely ergodicity under $\mathbb{Z}^2$-translations) in \Cref{thm:main2}, then many parts of the proof simplify.  In particular, \Cref{lem:a.s.constant} becomes immediate and in several places we can avoid appealing to ergodic decompositions.

\section{The random-cluster model}\label{sec:rc}
Theorem~\ref{thm:main} has implications for the well-studied random-cluster model.  Let $G$ be a finite graph with a designated \emph{boundary} $\partial G \subseteq V(G)$.  By \emph{boundary conditions} on $G$ we mean a partition $\xi$ of $\partial G$.  For a subgraph $\omega \in \{0,1\}^{E(G)}$ of $G$ with boundary conditions $\xi$, write $|\omega|$ for the number of open edges in $\omega$, and write $\kappa(\omega)$ for the number of connected components (of vertices) in $\omega$ after identifying the vertices in each part of $\xi$ (i.e., wiring them together).  The \emph{random-cluster model} on $(G,\xi)$ with parameters $p \in (0,1)$ and $q\in (0,\infty)$ is the percolation measure $\phi_{G,\xi,p,q}$ on $G$ defined by
$$\phi_{G,\xi,p,q}(\omega):=Z^{-1}\left(\frac{p}{1-p}\right)^{\abs{\omega}} q^{\kappa(\omega)},$$
where $Z=Z(G,\xi,p,q)>0$ is a normalizing factor to make $\phi_{G,\xi,p,q}$ into a probability measure.  The random-cluster model provides a simultaneous generalization for several notorious models in statistical physics, including Bernoulli percolation (from $q=1$) and the Ising model (from $q=2$).  See \cite{grimmett2004random} for further background.

The random-cluster model satisfies the FKG inequality exactly when $q \geq 1$, so it is no surprise that the $q \geq 1$ regime is much better understood than the $q<1$ regime.  For example, when $q \geq 1$ (and $p$ is arbitrary), one can obtain a random-cluster model on $\Z^2$ as a weak limit of the corresponding random-cluster models on finite grids, with either free or wired boundary conditions; in fact one does not need to pass to a subsequence to obtain convergence.  Such a result is not known for any choice of boundary conditions when $q<1$.  See \cite[Section 3.6]{grimmett2004random} and the recent work \cite{beffara2025newboundcriticalpoint}.

The main interest is in the locations of phase transitions as $p$ increases.  On the square lattice $\Z^2$, the \emph{self-dual point} is $p_{\sd}(q): =\frac{\sqrt{q}}{{1+\sqrt{q}}}$. In the positive-association regime $q \geq 1$, Zhang's non-coexistence argument shows that the critical point always is greater than or equal to the self-dual point.  Thanks to work of Beffara and Duminil-Copin~\cite{beffara2012self} and of Duminil-Copin, Raoufi, and Tassion~\cite{DCsharpness}, for $q \geq 1$ it is now known that the self-dual point is the critical point of a unique sharp phase transition (for all choices of boundary conditions); this generalizes Kesten's classical result \cite{kesten1980critical} for Bernoulli percolation ($q=1$).  The self-dual point is still conjectured to be critical in the $q<1$ regime, but not even the analogue of Zhang's result is known.  Our Theorem~\ref{thm:main} provides a weak substitute, as follows.

Consider the \emph{torus graph} $\mathbb{T}_n$, namely, the Cayley graph of $(\Z/n\Z)^2$ with respect to the generators $(1,0),(0,1)$ (with no boundary).  For each $n \in \Z_{\geq 3}$, this graph is self-dual, just like $\Z^2$. 
We can also interpret $\mathbb{T}_n$ as the grid graph on $[-n/2,n/2)^2$ with suitable boundary conditions.  Extracting a subsequence, we obtain (for arbitrary $q$) a percolation measure $\phi_{p,q}$ on $\Z^2$ as a weak limit of the measures $\phi_{\mathbb{T}_n,p,q}$; the dual measures $\phi_{\mathbb{T}_n,p,q}^*$ also weakly converge to the dual $\phi_{p,q}^*$.  Both $\phi_{p,q}$ and $\phi_{p,q}^*$ are translation-invariant (due to the symmetry of $\T_n$) and have finite-energy.  Theorem~\ref{thm:main}, applied to the individual components of the ergodic decomposition, yields the following corollary.

\begin{corollary}\label{cor:rc}
Let $p \in (0,1)$ and $q \in (0,\infty)$, and let $(\phi_{p,q})_{\hp},(\phi_{p,q}^*)_{\hp}$ be the marginals on $\Z \times \Z_{\geq 0}$ of the random-cluster measures defined above.  Then almost surely the number of infinite clusters in $(\phi_{p,q})_{\hp}$ plus the number of infinite clusters in $(\phi_{p,q}^*)_{\hp}$ is at most $1$. 
\end{corollary}

This result is the first of its type for $q<1$ and provides weak evidence that critical point is greater than or equal to the self-dual point. 
For $q\geq 1$, it is known that percolation in the half-plane occurs whenever $p > p_{\sd}$ (due to FKG, the Russo--Seymour--Welsh argument, and sharpness; see, e.g., \cite{duminil2019renormalization}).  One can obtain versions of Corollary~\ref{cor:rc} for other weak limits of random-cluster measures on grids (with various boundary conditions), provided that one performs extra averaging to ensure that the resulting limit is translation-invariant; we leave the details to the interested reader.

\section{The uniform odd subgraph}\label{sec:odd}
We explain how to adapt the arguments from Section~\ref{sec:thm1.1} to the uniform odd subgraph.
\subsection{Setup}
We begin with some general considerations.  A locally finite graph is \emph{even} (respectively, \emph{odd}) if all of its vertices have even (respectively, odd) degrees.  Every graph has even subgraphs (e.g., the empty subgraph), but only some graphs have odd subgraphs; a sufficient (but not necessary) criterion is admitting a perfect matching.

The \emph{uniform even subgraph} of a finite graph $G$ is the uniform distribution on the even subgraphs of $G$, viewed as an edge percolation measure.  In general, the set of even subgraphs forms a closed subgroup of $\{0,1\}^{E(G)} \cong (\Z/2\Z)^{E(G)}$, and the uniform even subgraph measure coincides with the Haar measure on this subgroup; this notion also makes sense when $G$ is infinite (and locally finite).

If a finite graph $G$ has odd subgraphs, then the \emph{uniform odd subgraph} of $G$ is the uniform distribution on the odd subgraphs of $G$.  One way to sample a uniform odd subgraph is to take the symmetric difference (XOR) of a fixed odd subgraph and a uniform even subgraph; this construction also provides a uniform odd subgraph measure for an infinite graph $G$.

We now specialize to the square lattice $\Z^2$.  The Haar measure approach produces uniform even and odd subgraph measures $\UEG, \UOG$ on $\Z^2$ (note that $\Z^2$ admits a perfect matching). 
These measures can also be constructed as suitable weak limits of the uniform even and odd subgraph measures on finite grids; it is easy to show that these two definitions are equivalent.  We mention a third equivalent construction: Assign signs $\{-,+\}$ independently and uniformly at random to the plaquettes of $\Z^2$, and declare an edge to be open if and only if its two adjacent plaquettes have the same sign.  This produces the uniform even subgraph, and one can obtain the uniform odd subgraph by taking the symmetric difference with any fixed ``reference'' perfect matching.  

It is obvious that $\UEG$ and $\UOG$ are both translation-invariant.  The third characterization makes it clear that they are both \emph{$2$-dependent}, in the sense that if $E_1, E_2$ are disjoint edge sets with no common vertices (i.e., the distance between $E_1,E_2$ is at least $2$), then events depending only on $E_1$ are independent from events depending only on $E_2$.  It follows that both measures are ergodic under the translation action of $\Z^2$, and in fact ergodic with respect to translations along any $1$-dimensional subgroup of $\Z^2$.\footnote{Indeed, for any $k \in \Z_{\geq 1}$, $k$-dependence implies mixing with respect to any $1$-dimensional subgroup of translations: Any two events can be approximated to arbitrary accuracy by cylinder sets each depending on only a finite set of edges, and these two edge sets can be moved arbitrarily far apart by applying a large enough translation.  Mixing is a stronger property than ergodicity.}  Because of the strict local degree conditions, these measures do not have finite-energy or satisfy the FKG inequality.

A final important property is that $\UEG, \UOG$ are both self-complementary (in fact this is the case on any even graph $G$).

\subsection{Previous work}
It is known \cite{GMM18,hansen2023uniform} that for all $d \geq 2$, the uniform even subgraph of $\Z^d$ percolates (has infinite clusters). The argument for $d\geq 3$, which is based on a comparison with Bernoulli-$1/2$ percolation in a hyperplane, also applies to the uniform odd subgraph. The only remaining open case is the odd subgraph on $\Z^2$, which is expected not to percolate (see \cite{464445} for representative pictures).

This last problem, posed in \cite{hansen2022strict}, was motivated by the study of the Kert\'esz line for the loop $O(1)$ model.  The loop $O(1)$ model is supported on the set of even subgraphs.  Adding a ``ghost vertex'' adjacent to all other vertices amounts to introducing the analogue of a magnetic field. Conditioning on all edges incident to the ghost vertex being open (corresponding to a strong magnetic field) imposes the condition that all remaining (internal) vertices have odd degrees. Non-percolation for the uniform odd subgraph of $\Z^2$ would highlight a difference in behavior between the Kert\'esz line for this model and the Kert\'esz line for the random-cluster model, which is monotone \cite{hansen2022strict}.


\subsection{Half-plane non-coexistence}

Our proof of Theorem~\ref{thm:main2} can be modified for the setting of the uniform odd subgraph of $\Z^2$, even though $\UOG$ does not have finite-energy.  We establish half-plane non-coexistence for the primal and the complement (rather than the primal and the dual); self-complementarity then implies that the uniform odd subgraph does not percolate in the half-plane.  The key point is the geometric/topological fact that the primal and the complement cannot ``cross'';\footnote{For general planar graphs, the relevant ``non-crossing'' condition is that at each vertex the open edges form an interval with respect to the cyclic ordering given by the planar embedding. The same principle explains why corner percolation on $\Z^2$ always consists of disjoint loops (or bi-infinite paths) \cite{pete2008corner}, and it was essential in the exploration of the random-current backbone in \cite{klausen2022mass}. } see Figure~\ref{fig:odd_cannot_cross}.
The primal and the complement can certainly cross in even subgraphs, and this difference in ``permeability'' explains why one would expect large clusters to coexist more easily in the uniform even subgraph than in the uniform odd subgraph.  It is shown in \cite{hansen2023uniform} (see the proof of Proposition 2.9 and Theorem 2.10 there) that $\UEG_{\hp}$ does percolate, and one can show that almost surely there is a unique infinite cluster.

\begin{figure}
\centering
\begin{tikzpicture}[
  scale=0.7,
  vertex/.style={circle,fill=black,inner sep=1.2pt},
  full/.style={line width=1.1pt},
  dashededge/.style={line width=1.1pt,densely dashed},
  faint/.style={gray!40,line width=0.6pt},
  lab/.style={font=\small}
]

\begin{scope}[shift={(0,0)}]

  \draw[faint] (-1,0) -- (1,0);
  \draw[faint] (0,-1) -- (0,1);

  \node[vertex] at (0,0) {};
  \node[vertex] at (1,0) {};
  \node[vertex] at (-1,0) {};
  \node[vertex] at (0,1) {};
  \node[vertex] at (0,-1) {};

  \draw[dashededge] (0,-1) -- (0,0) -- (0,1);

  \draw[full] (-1,0) -- (0,0) -- (1,0);


\end{scope}


\begin{scope}[shift={(5,0)}]

  \draw[faint] (-1,0) -- (1,0);
  \draw[faint] (0,-1) -- (0,1);

  \node[vertex] at (0,0) {};
  \node[vertex] at (1,0) {};
  \node[vertex] at (-1,0) {};
  \node[vertex] at (0,1) {};
  \node[vertex] at (0,-1) {};

  \draw[dashededge] (0,0) -- (1,0);

  \draw[full] (0,0) -- (0,1);
  \draw[full] (0,0) -- (0,-1);
  \draw[full] (0,0) -- (-1,0);

\end{scope}


\end{tikzpicture}
\caption{The open (solid edges) and complement (dashed edges) can ``cross'' only if they are arranged as on the left, where the central vertex has degree $2$.  There can be no such ``crossings'' in an odd subgraph, since each vertex is as depicted on the right. \label{fig:odd_cannot_cross}}
\end{figure}
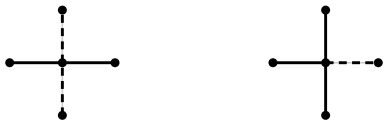

\begin{lemma}\label{lem:UOG-BK}
$\operatorname{UOG}$ almost surely has only finitely many infinite clusters. 
\end{lemma}

\begin{proof}
We adapt the classical argument of Burton and Keane.  Their original argument uses finite-energy to construct so-called \emph{trifurcation points} in boxes.  A suitable substitute replaces trifurcation points with \emph{coarse trifurcation regions} (which were also used in \cite{aizenman2015random}); a sufficient statement is that if $B \subseteq \Z^2$ is a box that with positive probability touches at least $9$ infinite clusters of the marginal on $\Z^2 \setminus B$, then with positive probability some $3$ such clusters glue together when we reveal the subgraph on $B$.  Suppose that at least $9$ infinite clusters of $\Z^2 \setminus B$ touch $B$.  By the pigeonhole principle, some side of $B$ intersects at least $3$ of these clusters.  By flipping some subset of the plaquettes inside of $B$ touching this side, we find that with positive probability (depending only on the size of $B$) all of the edges on this side are open and in particular the clusters intersecting this side all glue together.  The Burton--Keane volume argument now demonstrates that almost surely $\UOG$ has only finitely many infinite clusters.
\end{proof}
We remark that one can easily adapt the rest of the Burton--Keane argument to deduce that $\UOG$ either almost surely has no infinite cluster, or almost surely has a unique infinite cluster which is moreover almost surely $1$-ended.

With the previous lemma in hand, we can easily modify the argument of Theorem~\ref{thm:main2}.  The argument is simpler here than in the setting of \Cref{proposion31} in the sense that we do not have to worry about our paths to infinity going ``below'' the origin plaquette, and it is more complicated in the sense that we have to take care with how the paths glue to the $x$-axis.

\begin{theorem}\label{thm:odd}
Almost surely $\UOG_{\hp}$ has no infinite cluster.
\end{theorem}

\begin{proof}
Lemmas~\ref{lem:UOG-BK} and~\ref{lem:a.s.constant} together imply that the number of infinite clusters in $\UOG_{\hp}$ is either almost surely $0$ or almost surely $1$.  We will show that if we assume that $\UOG_{\hp}$ almost surely has a unique infinite cluster, then almost surely its complement has no infinite cluster; since $\UOG_{\hp}$ is self-complementary, this will complete the proof that $\UOG_{\hp}$ almost surely has no infinite cluster.

Assume that $\UOG_{\hp}$ almost surely has a unique infinite cluster.  It suffices to show that almost surely the origin is not contained in an infinite cluster in the complement.  The argument from the first paragraph of the proof of Proposition~\ref{proposion31} shows that the origin is contained in an infinite (primal) cluster with positive probability.  Flipping the plaquettes centered at $(-1/2,-1/2)$ and $(1/2,-1/2)$ toggles the horizontal edges incident to the origin, and the states of these edges are independent of all other edges in $\UOG_{\hp}$ (and of one another).  Since opening these edges cannot disconnect the origin from infinity, we see that with positive probability the origin is contained in an infinite cluster and both of the horizontal edges incident to the origin are open.  By Lemma~\ref{lem:translation-invariant-events}, almost surely there are some $x_1<0$ and $x_2>0$ such that each $(x_i,0)$ is contained in the unique infinite cluster and both of the horizontal edges incident to each $(x_i,0)$ are open.   Thus there is a path from $(x_1,0)$ to $(x_2,0)$, and this path together with the $x$-axis form a loop that prevents the origin from escaping to infinity in the complement.  (The extra horizontal edges incident to $(x_i,0)$ ensure that the complement cannot ``leak out'' at the boundary between the $x$-axis and the path from $(x_1,0)$ to $(x_2,0)$.)
\end{proof}

\section{Concluding remarks} \label{sec:conclusion}

So far we have been concerned with non-coexistence phenomena for bond percolation in the half-plane.  We remark, as an aside, that our methods should also have ramifications for site percolation; we leave this as a topic for future work.  The purpose of this final section is to present several questions and examples about qualitative differences between the half-plane and full-plane settings.  A great variety of counterintuitive behavior turns out to be possible in the full plane, which makes it all the more remarkable that the half-plane admits simple true statements like Theorems~\ref{thm:main} and~\ref{thm:main2}.

The authors of \cite{haggstrom2009some} describe an intuition based on \cite{grimmett1981critical} for why (in contrast with Theorem~\ref{thm:main}) translation-invariance and finite-energy are not sufficient for non-coexistence in the full plane: One could imagine a one-ended spiral, similar to the uniform spanning tree, where the thickness of the branches grows fast enough to survive iid thinning (the source of finite-energy).  The ``backbone'' of such a spiral would cross every line infinitely often, so the marginal on each half-plane would have only finite clusters in both the primal and the dual.  It is conceivable that examples of this sort are the main (perhaps only) reason that one could have non-coexistence in half-planes but coexistence in the full plane.


\begin{question}\label{conj:k-dep}
Is it true that if $\mu$ is a translation-invariant, ergodic, finite-energy, $k$-dependent edge percolation measure on $\Z^2$ (for some $k \in \mathbb{Z}_{\geq 1}$), then the trichotomy of Theorem~\ref{thm:main} holds in the full plane?
\end{question}

The following example of the random walk web shows that the finite-energy hypothesis in this question cannot be removed.  The random web walk was first introduced by Scheidegger \cite{SCHEIDEGGER01031967} and Arratia \cite{Arratia1981Coalescing} and later shown to be related to the self-repelling motion \cite{toth_true_1998} and the Brownian web \cite{fontes2004brownian}. A sample is shown in \Cref{fig:full_plane_examples}.

\begin{figure}
    \centering
    \begin{subfigure}[b]{0.48\linewidth}
        \centering
        \includegraphics[width=\linewidth]{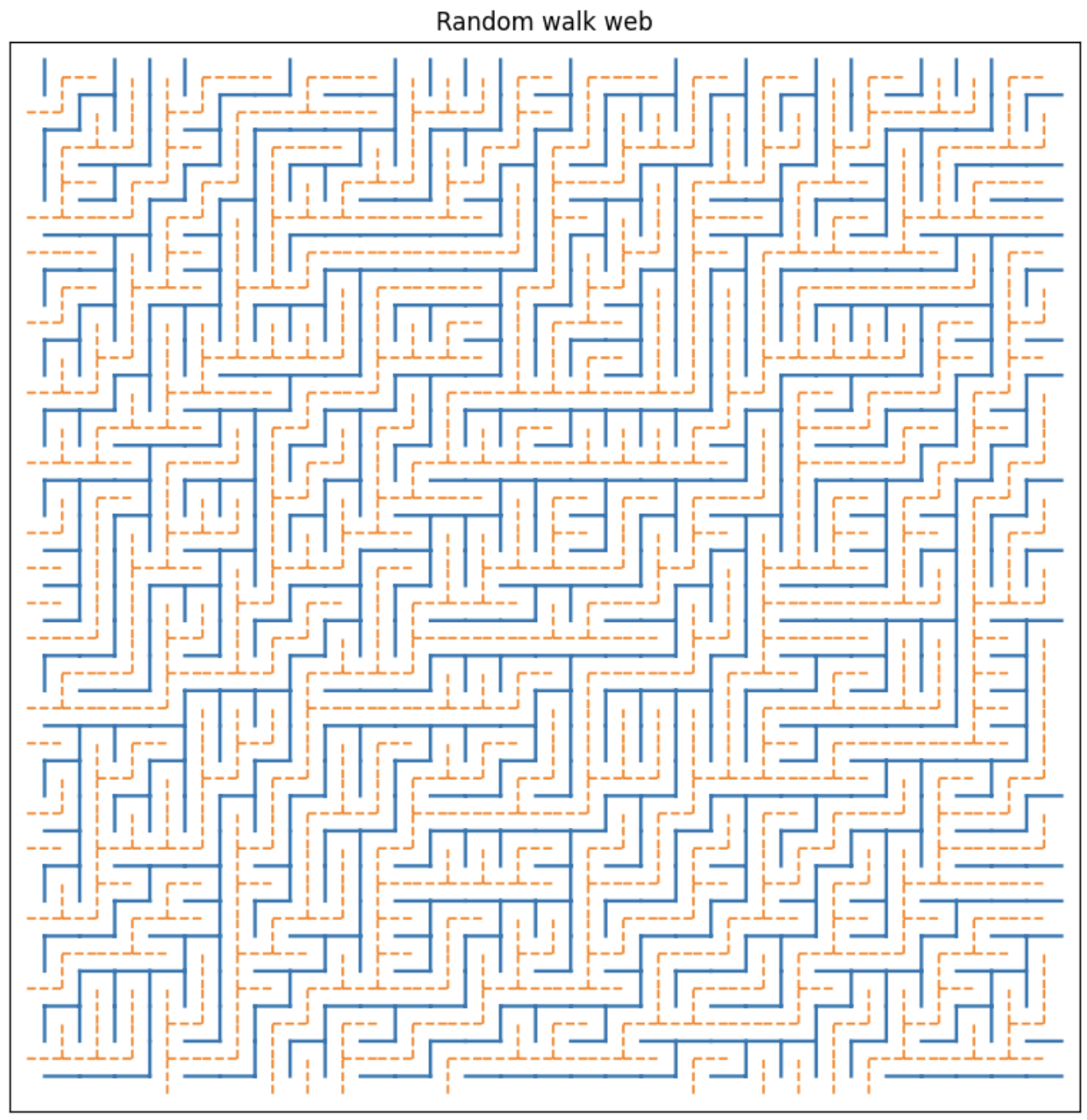}
        \caption{The random walk web discussed in \Cref{ex:random_walk_web}. This model is 2-dependent and has coexistence of primal and dual infinite clusters.}
        \label{fig:random_walk_web_a}
    \end{subfigure}
    \hfill
    \begin{subfigure}[b]{0.48\linewidth}
        \centering
        \includegraphics[width=\linewidth]{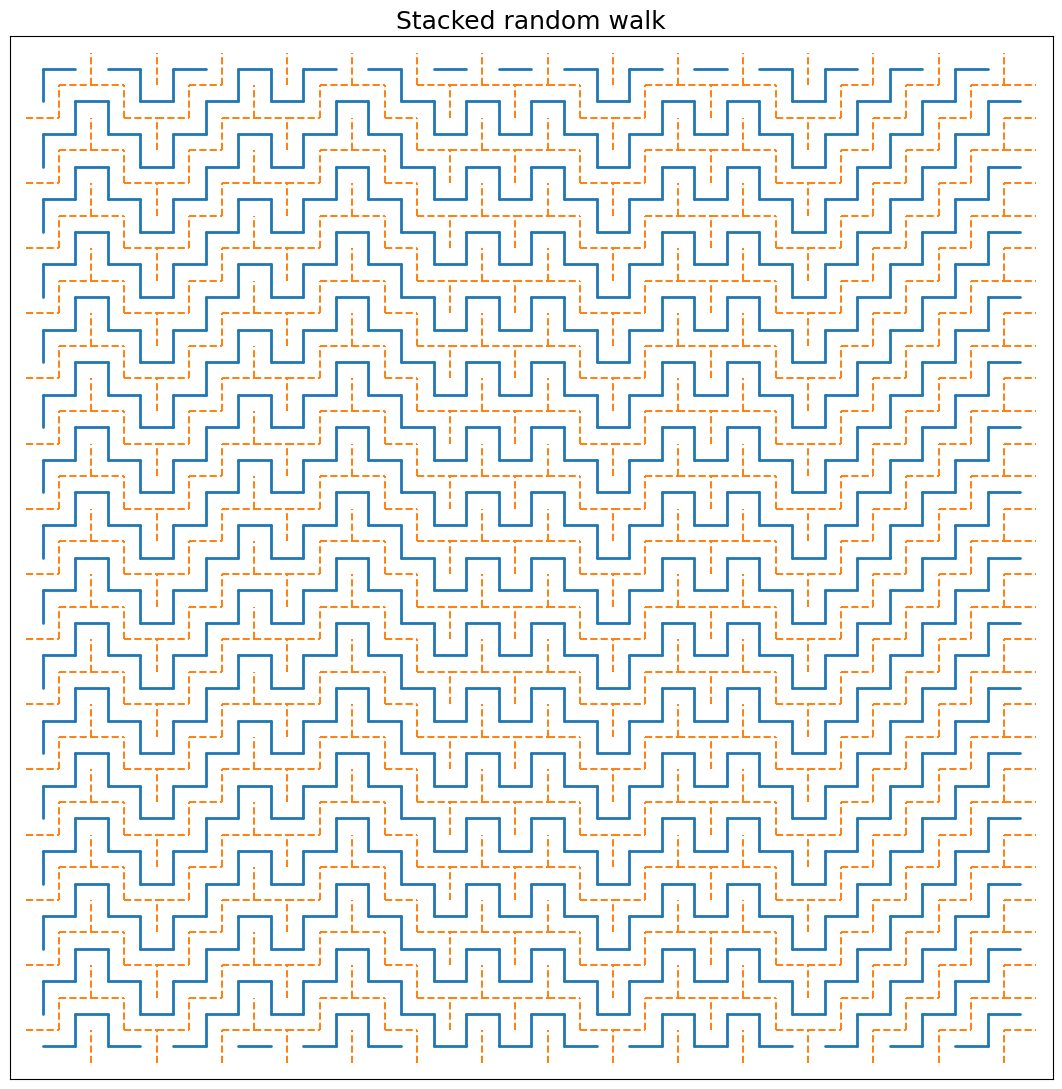}
        \caption{The stacked random walk defined in \Cref{ex:stacked_random_walks}. It has infinitely many infinite clusters in the full plane but none in the upper half-plane.}
        \label{fig:random_walk_web_b}
    \end{subfigure}
    \caption{Samples from two examples of non-finite energy percolation models in the full plane. }
    \label{fig:full_plane_examples}
\end{figure}

\begin{example}[Random walk web]\label{ex:random_walk_web}
To sample a subgraph $H$ from the random web walk, first sample a uniformly random iid array $(\epsilon_{x,y})_{(x,y) \in \mathbb{Z}^2} \in \{-1,1\}^{\mathbb{Z}^2}$.  For each $(x,y) \in \mathbb{Z}^2$, open the vertical edge from $(x,y)$ to $(x,y+1)$ in $H$ if $\epsilon_{x,y}=1$, and open the horizontal edge from $(x,y)$ to $(x,y+1)$ in $H$ if $\epsilon_{x,y}=-1$.

The resulting measure is $1$-dependent, and it is known that $H$ is almost surely a spanning tree (and in particular has a unique infinite cluster in the full plane).  It follows that the same holds for the dual $H^*$.  
\end{example}

With a view towards applications like the antiferromagnetic Ising model, one might hope to replace the $k$-dependence condition in Question~\ref{conj:k-dep} with the domain Markov property.

In a different direction, Lemma~\ref{lemma:one_end_to_no_clust} seems potentially useful for extending half-plane non-coexistence results to other models.  One that comes to mind is the arboreal gas, whose dual is Bernoulli percolation conditioned to be connected.  The only known proofs of non-percolation for the arboreal gas in the full plane rely on the complicated supersymmetry methods of \cite{Bauerschmidt2021PlaneForest}.  There is a non-supersymmetric proof \cite{halberstam2024uniqueness} that any infinite cluster must be unique and $1$-ended (of course there is no such cluster by \cite{Bauerschmidt2021PlaneForest}), and this can be used as input to \Cref{thm:main2} for further supersymmetry-free results.  
Another promising topic for future work in this direction is the edge-weighted planar odd graph.   It would also be natural to try to relax the conditions of Lemma~\ref{lemma:one_end_to_no_clust}.
\begin{question}
Is it true that any translation-invariant, $k$-dependent percolation measure on $\Z^2$ with single edge marginals in $(0,1)$ almost surely has at most $1$ infinite cluster?
\end{question}
The work \cite[Theorem 0.0]{liggett1997domination} provides some $0<p(k)<1$ such that if the edge marginals are larger than $p(k)$, then the process dominates Bernoulli percolation at $p=0.51$, whence uniqueness follows.

Finally, we revisit the hypothesis in \Cref{thm:main2} that $\mu$ almost surely has only finitely many infinite clusters in the full plane.  This hypothesis is certainly necessary for half-plane non-coexistence, since for the simple example of the measure $\mu$ which deterministically has all vertical edges open and all horizontal edges closed, infinitely many infinite clusters survive in both $\mu_{\hp}$ and $ \mu^*_{\hp}$.  At the same time, the following example, which we found using ChatGPT 5.5 Pro, shows that it is possible for $\mu$ to almost surely have infinitely many infinite clusters and for $\mu_{\hp}$ and $ \mu^*_{\hp}$ to almost surely have only finite clusters. A sample from the measure is shown in \Cref{fig:full_plane_examples}.

\begin{example}[Stacked random walk]\label{ex:stacked_random_walks}
We define an edge percolation $\mu$ on $\mathbb{Z}^2$ as follows.  To sample a subgraph $H$ from $\mu$, first sample a uniformly random iid sequence $(\epsilon_n)_{n \in \mathbb{Z}} \in \{-1,1\}^{\mathbb{Z}}$.  Define the function $f: \mathbb{Z} \to \mathbb{Z}$ by sampling $f(0) \in \{0,1\}$ uniformly at random and then extending it to satisfy $f(n+1)=f(n)+\epsilon_n$ for all $n \in \mathbb{Z}$ (so $f(n)$ is a $2$-sided unbiased random walk starting from $f(0)$).  Now for each choice of $n,m \in \mathbb{Z}$, open the vertical edge from $(n,f(n)+2m)$ to $(n,f(n+1)+2m)$ and the horizontal edge from $(n,f(n+1)+2m)$ to $(n+1,f(n+1)+2m)$ in $H$.

Notice that $H$ is always the disjoint union of infinitely many bi-infinite paths (spaced apart vertically at distance $2$).  In the upper half-plane, however, $H$ almost surely has only finite clusters, since almost surely $\liminf_{n \to \infty} f(n)=\liminf_{n \to -\infty}f(n)=-\infty$.
\end{example}

It would be interesting to construct a version of this example where the marginal on \emph{each} half-plane (not only the upper and lower half-planes) almost surely has only finite clusters.

\appendix
\section{Using finite-energy to eliminate tenuous clusters}\label{sec:appendix}

In Lemma~\ref{lemma:one_end_to_no_clust} and Proposition~\ref{lemma:uniqueness_implies_non_tenuousness}, we used topological arguments to show that if $\mu$ is a translation-invariant edge percolation measure that almost surely has finitely many infinite clusters, then almost surely $\mu_{\hp}$ has no tenuous infinite cluster.  In this appendix we provide a more ``combinatorial'' proof that finite-energy alone is sufficient condition to rule out tenuous infinite clusters. This argument can be used to prove \Cref{thm:main} without  Lemma~\ref{lemma:one_end_to_no_clust}.   We suspect that the proof ideas may be of use for related problems.

\begin{lemma}\label{lem:tenuous}
Let $\mu$ be a finite-energy edge percolation measure on $\Z^2$.  Then almost surely $\mu_{\hp}$ has no tenuous infinite cluster.
\end{lemma}

\begin{proof}
An $n$-tenuous infinite cluster necessarily contains some vertex with $y$-coordinate at most $n-1$.  By a union bound (and the fact that $n$-tenuousness implies $n'$-tenuousness for $n'>n$) it suffices to show that each vertex $(\ell,n-1) \in \Z \times \Z_{\geq 0}$ is almost surely not contained in an $n$-tenuous infinite cluster.  For ease of notation we prove this only for $\ell=0$.  Fix some $n \geq 1$.  The rough idea is that if $(0,n-1)$ is in an $n$-tenuous infinite cluster, then there must be infinitely many coordinates $m$ such that ``cutting'' along the line segment $\{m\} \times [0,n]$ disconnects $(0,n-1)$ from infinity; we will use the  finite-energy hypothesis to show that almost surely there can be only finitely many such cut-coordinates.

Sample a subgraph $H$ from $\mu_{\hp}$, and reveal it according to the following procedure.  Begin by revealing $H$ on
$$(\Z \times \Z_{\geq n}) \cup ((-1,1) \times [0,n)),$$
and let $C_0$ denote the cluster containing $(0,n-1)$ in the subgraph of $H$ revealed thus far.  If $C_0$ is infinite, then certainly the cluster in $H$ containing $(0,n-1)$ is not $n$-tenuous.  Also, if $C_0$ is finite and touches neither of the half-strips $(-\infty, -1] \times [0,n], ~[1,\infty) \times [0,n]$, then the cluster in $H$ containing $(0,n-1)$ is necessarily finite and thus not $n$-tenuous.  If either of these two outcomes occurs, then halt the procedure.

Suppose neither of the last two outcomes occurs.  Let $m^+_1$ denote the largest $m\geq 0$ such that $C_0$ touches the segment $\{m\} \times [0,n]$, and let $m^-_1$ denote the smallest (i.e., most negative) $m\leq 0$ such that $C_0$ touches the segment $\{m\} \times [0,n]$.  Our assumptions ensure that
$$0<m^+_1+(-m^-_1)<\infty.$$
Reveal $H$ on the two rectangles
$$(m^-_1,-1] \times [0,n), \quad [1,m^+_1) \times [0,n),$$
and then reveal $H$ on the two columns
$$(m^-_1-1,m^-_1] \times [0,n), \quad [m^+_1,m^+_1+1) \times [0,n).$$
Since $m^+_1,m^-_1$ are not both $0$, at least one of these two columns was not previously revealed.  Let  $C_1$ denote the cluster containing $(0,n-1)$ in the subgraph of $H$ revealed thus far.  As before, if $C_1$ is infinite or touches neither of the half-strips $(-\infty, -m^-_1-1] \times [0,n],~ [m^+_1+1,\infty) \times [0,n]$, then the cluster in $H$ containing $(0,n-1)$ is necessarily not $n$-tenuous, and we halt the procedure.

\begin{figure}[t]
    \centering
\begin{tikzpicture}[scale=0.7]

\def\n{4}  
\def\mplus{4}  
\def\mminus{-3}  
\def\xmin{-6}
\def\xmax{7}
\def\ymax{7}

\def\ymin{-0.1}
\def\ytop{\ymax+0.5}
\def\xleft{\xmin-0.5}
\def\xright{\xmax+0.5}


\fill[blue!8] (\xmin-0.5, \n-0.1) rectangle (\xmax+0.5, \ymax+0.5);

\fill[blue!8] (\mminus-0.5, -0.1) rectangle (\mplus+0.5, \n);

\fill[yellow!15] (\xmin-0.5, -0.1) rectangle (\mminus-0.5, \n-0.1);
\fill[yellow!15] (\mplus+0.5, -0.1) rectangle (\xmax+0.5, \n-0.1);


\foreach \x in {\xmin,...,\xmax} {

    \draw[gray!30] (\x, \n) -- (\x, \ytop);

    \pgfmathparse{(\x >= \mminus && \x <= \mplus) ? 1 : 0}
    \ifnum\pgfmathresult=1
        \draw[gray!30] (\x, \ymin) -- (\x, \n);
    \else
        \draw[gray!15] (\x, \ymin) -- (\x, \n);
    \fi
}

\foreach \y in {0,...,\ymax} {
    \ifnum\y<\n
        \draw[gray!15] (\xleft, \y) -- (\mminus, \y);
        \draw[gray!30] (\mminus, \y) -- (\mplus, \y);
        \draw[gray!15] (\mplus, \y) -- (\xright, \y);
    \else
        \draw[gray!30] (\xleft, \y) -- (\xright, \y);
    \fi
}


\fill[orange!30] (\mplus-0.1, -0.1) rectangle (\mplus+1-0.1, \n-0.1);
\foreach \y in {1,...,\n} {
    \draw[orange, line width=2pt] (\mplus, \y-1) -- (\mplus+1, \y-1);
    \draw[orange, line width=2pt] (\mplus, \y-1) -- (\mplus, \y);
}

\fill[orange!30] (\mminus-1+0.1, -0.1) rectangle (\mminus+0.1, \n-0.1);
\foreach \y in {1,...,\n} {
    \draw[orange, line width=2pt] (\mminus-1, \y-1) -- (\mminus, \y-1);
    \draw[orange, line width=2pt] (\mminus, \y-1) -- (\mminus, \y);
}



\draw[red!70!black, line width=2.5pt] (0,\n-1) -- (0,\n);
\draw[red!70!black, line width=2.5pt] (0,\n-1) -- (0,\n-2);
\draw[red!70!black, line width=2.5pt] (0,\n-2) -- (0,\n-3);

\draw[red!70!black, line width=2.5pt] (0,\n) -- (-1,\n);
\draw[red!70!black, line width=2.5pt] (-1,\n) -- (-1,\n+1);
\draw[red!70!black, line width=2.5pt] (-1,\n+1) -- (-2,\n+1);
\draw[red!70!black, line width=2.5pt] (-2,\n+1) -- (-2,\n+2);
\draw[red!70!black, line width=2.5pt] (-2,\n+2) -- (-3,\n+2);
\draw[red!70!black, line width=2.5pt] (-3,\n+2) -- (-3,\n+1);
\draw[red!70!black, line width=2.5pt] (-3,\n+1) -- (-3,\n);

\draw[red!70!black, line width=2.5pt] (0,\n) -- (1,\n);
\draw[red!70!black, line width=2.5pt] (1,\n) -- (1,\n+1);
\draw[red!70!black, line width=2.5pt] (1,\n+1) -- (2,\n+1);
\draw[red!70!black, line width=2.5pt] (2,\n+1) -- (2,\n);
\draw[red!70!black, line width=2.5pt] (2,\n) -- (2,\n-1);
\draw[red!70!black, line width=2.5pt] (0,\n-2) -- (0,1);
\draw[red!70!black, line width=2.5pt] (0,1) -- (1,1);
\draw[red!70!black, line width=2.5pt] (1,1) -- (2,1);
\draw[red!70!black, line width=2.5pt] (2,1) -- (3,1);
\draw[red!70!black, line width=2.5pt] (3,1) -- (4,1);
\draw[red!70!black, line width=2.5pt] (0,2) -- (-1,2);


\foreach \x in {\xmin,...,\xmax} {
    \foreach \y in {0,...,\ymax} {
        \fill[black!50] (\x,\y) circle (1.5pt);
    }
}

\fill[red!70!black] (0,\n-1) circle (4pt);
\node[below left, red!70!black] at (0,\n-1) {$(0,n{-}1)$};

\foreach \pos in {(0,\n), (0,\n-2), (0,\n-3), (-1,\n), (-1,\n+1), (-2,\n+1), (-2,\n+2), (-3,\n+2), (-3,\n+1), (-3,\n)} {
    \fill[red!70!black] \pos circle (3pt);
}
\foreach \pos in {(1,\n), (1,\n+1), (2,\n+1), (2,\n), (2,\n-1), (0,1), (1,1), (2,1), (3,1), (4,1),(-1,2)} {
    \fill[red!70!black] \pos circle (3pt);
}


\fill[blue!70!black] (\mminus,\n) circle (4pt);
\fill[blue!70!black] (\mplus,1) circle (4pt);


\draw[blue!50!black, dashed, line width=1pt] (\xmin-0.5, \n) -- (\xmax+0.5, \n);
\node[left, blue!50!black] at (\xmin-0.6, \n) {$y = n$};

\draw[black!50, dashed] (\xmin-0.5, 0) -- (\xmax+0.5, 0);
\node[left, black!50] at (\xmin-0.6, 0) {$y = 0$};

\draw[<-, >=Stealth, thick, orange!70!black] (\mplus, -0.4) -- (\mplus, -1.1);
\node[below, orange!70!black] at (\mplus, -1.1) {$m_i^+$};

\draw[<-, >=Stealth, thick, orange!70!black] (\mminus, -0.4) -- (\mminus, -1.1);
\node[below, orange!70!black] at (\mminus, -1.1) {$m_i^-$};


\node[red!70!black, align=center] at (-4.1, \n+2.4) {$C_{i-1}$};
\draw[->, >=Stealth, red!70!black] (-3.5, \n+2.2) -- (-3.2, \n+1.6);

\end{tikzpicture}
    \caption{One step in the exploration of the cluster containing the vertex $(0,n-1)$.  The blue region has been previously revealed, and the yellow region is still unexplored.  If all of the orange edges are closed, then the cluster gets ``cut off'' and must be finite.}
    \label{fig:exploration}
\end{figure}
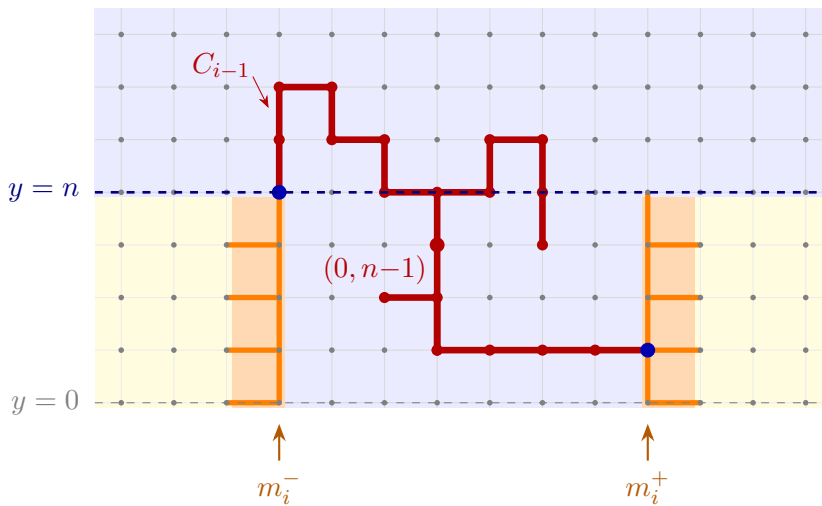

We continue in the same manner to obtain a sequence of triples $(m^+_i, m^-_i, C_i)$.  For the general step, suppose that we have already obtained $m^-_i \leq 0 \leq m^+_i$, revealed $H$ on
$$(\Z \times \Z_{\geq n}) \cup ((m^-_i-1,m^+_i+1) \times [0,n)),$$
and set $C_i$ to be the cluster containing $(0,n-1)$ in the subgraph of $H$ revealed thus far.  Say that the index $i$ is \emph{halting} if $C_i$ is infinite or touches neither of the half-strips $(-\infty, -m^-_i-1] \times [0,n],~ [m^+_i+1,\infty) \times [0,n]$.  As the name suggests, we halt the procedure at the index $i$ if it is halting.  If the index $i$ is not halting, then we obtain $(m^+_{i+1}, m^-_{i+1},C_{i+1})$ as follows.  Let $m^+_{i+1}$ denote the largest $m\geq 0$ such that $C_i$ touches the line segment $\{m\} \times [0,n]$, and let $m^-_{i+1}$ denote the smallest $m\leq 0$ such that $C_i$ touches the line segment $\{m\} \times [0,n]$.  Our non-halting assumption ensures that
\begin{equation}\label{eq:sum-of-m's-increase}
m^+_i+(-m^-_i)<m^+_{i+1}+(-m^-_{i+1})<\infty.
\end{equation}
Reveal $H$ on the two rectangles
$$(m^-_{i+1},m^-_i-1] \times [0,n), \quad [m^+_i+1,m^+_{i+1}) \times [0,n),$$
and then reveal $H$ on the two columns
\begin{equation}\label{eq:columns}
(m^-_{i+1}-1,m^-_{i+1}] \times [0,n), \quad [m^+_{i+1},m^+_{i+1}+1) \times [0,n).
\end{equation}
Notice from \eqref{eq:sum-of-m's-increase} that at least one of these two columns was not previously revealed.   Let  $C_{i+1}$ denote the cluster containing $(0,n-1)$ in the subgraph of $H$ revealed thus far.  See Figure~\ref{fig:exploration} for an illustration.

The  finite-energy assumption provides some $\gamma>0$ (depending only on $\mu,n$; and independent of $i$) such that almost surely given what has been previously revealed, when we reveal $H$ on \eqref{eq:columns}, with probability at least $\gamma$ all of these edges are closed.  It follows that with probability at least $\gamma$, the index $i$ is halting.  Thus the probability that the indices $1,2,\ldots, j$ are all non-halting is at most $(1-\gamma)^j$.  Since this quantity tends to $0$ as $j$ tends to infinity, we conclude that the procedure almost surely halts. Thus the cluster in $H$ containing $(0,n-1)$ is almost surely not $n$-tenuous, as desired.
\end{proof}

\section*{Acknowledgments}
We thank Omer Angel, Geoffrey Grimmett, Renan Gross, Ron Peled, and Ulrik Thinggaard Hansen for helpful conversations.  FRK was supported by the Carlsberg Foundation, CF24-0466.  NK was supported in part by a NSF Mathematical Sciences Postdoctoral Research Fellowship under grant DMS-2501336.

\bibliographystyle{abbrv}
\bibliography{bibliography}

\end{document}